\documentclass
{amsart}

\usepackage[english]{babel}
\usepackage{amsmath,amsfonts,amssymb,amsthm,enumerate,hyperref}
\usepackage{latexsym}
\usepackage{amscd}

\numberwithin{equation}{section}

\newtheorem{definition}{Definition}[section]
\newtheorem{lemma}[definition]{Lemma}
\newtheorem{theorem}[definition]{Theorem}
\newtheorem{corollary}[definition]{Corollary}

\newtheorem{proposition}[definition]{Proposition}

\newtheorem{em-example}[definition]{Example}
\newtheorem{em-def}[definition]{Definition}
\newtheorem{em-remark}[definition]{Remark}
\newtheorem{em-question}[definition]{Question}

\newenvironment{example}{\begin{em-example} \em }{ \end{em-example}}
\newenvironment{remark}{\begin{em-remark} \em }{\end{em-remark}}

\newcommand{\R}{\mathbb R}
\newcommand{\N}{\mathbb N}
\newcommand{\C}{\mathcal C}

\newcommand{\Z}{\mathbb Z}

\def\mod{\mathbf{Mod}}

\def\ent{\mathrm{ent}}

\def\End{\mathrm{End}}

\def\mod#1{\mathrm{Mod}(#1)}

\def\F{\mathcal F}

\def\lFin{\mathrm{lFin}}

\def\Fin{\mathrm{Fin}}
\def\Ker{\mathrm{Ker}}
\def\ker{\mathrm{Ker}}

\def\hom{\mathrm{Hom}}
\def\mult{\mathrm{mult}}

\title[Addition Theorem for non-discrete entropy]{The Addition Theorem for algebraic entropies induced by non-discrete length functions}

\author{Luigi Salce}
\address{Dipartimento di Matematica, Via Trieste 63, 35121 Padova, Italy}
\email{salce@math.unipd.it}
\author{Simone Virili}
\email{virili@math.unipd.it}

\thanks{Research supported by ``Progetti di Eccellenza 2011/12" of Fondazione CARIPARO}

\begin{document}

\maketitle

\begin{abstract}
The validity of the Addition Theorem for algebraic entropies $\ent_L$ induced by non-discrete length functions $L$ on the category of locally $L$-finite modules over arbitrary rings is proved. Concrete examples of non-discrete length functions and their induced algebraic entropies are provided.
\end{abstract}


\section{Introduction}

Probably the most important result in the theory of algebraic entropies, and certainly the most difficult to prove, is the Addition Theorem (AT, for short). It states that, given an endomorphism $\phi$ of an $R$-module $M$ satisfying suitable conditions, and a $\phi$-invariant submodule $N$, the following equality holds
$$\ent (\phi) = \ent(\phi \restriction_N) + \ent(\bar \phi),$$ 
where $\bar \phi$ is the map induced by $\phi$ on the factor module $M/N$. Here $R$ denotes an arbitrary ring and $\ent$ is an algebraic entropy however defined. The importance of AT relies on the fact that the value of the entropy of the endomorphism $\phi$ on the whole module $M$ can be realized as  the sum of the entropies of the induced maps on cyclic trajectories of $\phi$-invariant subsections (submodules of factors). We refer to \cite{DGSZ}, \cite{SZ}, \cite{SVV}, \cite{DG}, \cite{VI} and \cite{DGSV} for the proofs of AT for various different algebraic entropies and their discussion. For unexplained notions appearing above and along this paper we refer to \cite{SVV}.

The goal of this paper is to prove that AT holds for algebraic entropies $\ent_L$ induced by non-discrete length functions $L: {\rm Mod} (R) \to \R_{\geq 0} \cup \{ \infty \}$ on the subcategory ${\rm lFin}_L[X]$ of ${\rm Mod} (R)$ consisting of the locally $L$-finite modules. The motivation is that AT was proved in \cite{SVV} for ${\rm lFin}_L[X]$ only for algebraic entropies $\ent_L$ induced by discrete length functions $L$, which are discrete entropies indeed. As remarked in \cite{SVV}, the proof of that result ``strongly depends on the discreteness of $L$ since it makes use of inductive arguments on the values of $L$". 

Dealing with non-discrete length functions, a completely new approach to the proof of AT is needed. We  reach it via a series of reductions. First we show that the algebraic entropy $\ent_L(\phi)$ can be computed considering only submodules of finite $L$-length of $M$ which are finitely generated (Proposition \ref{sanatoria}). Then 
we prove a formula for $\phi$ an automorphism (Proposition \ref{nolimit}), which allows us to escape the limit calculation in the definition of $\ent_L$, and to give a direct proof of AT  for locally $L$-finite modules when $\phi$ is bijective (Theorem \ref{AT_for_inv}). Finally, we show how to reduce to the bijective case first factoring modulo the hyperkernel of $\phi$ (Proposition \ref{hyperkernel}), and then using  the tool of central localisation (Corollary \ref{coro_loc_AT}), which was introduced in \cite{G} and, in our particular case, corresponds to tensoring an $R[X]$-module by the ring $R[X,X^{-1}]$. As an application, we prove in Section \ref{Unique_section} a Uniqueness Theorem for $\ent_L$, where $L$ is the non-discrete length function induced on $\mod R$, for $R$ a non-discrete archimedean valuation domain.

One of the consequences of the AT proved in \cite{SVV} for discrete length functions is that, if the base ring is Noetherian, then $\ent_L$ coincides with the $L$-multiplicity symbol (which is the usual multiplicity symbol when $L=\ell$ is the composition length) on the class of locally $L$-finite modules. With this in mind, one may consider the entropy induced by a non-discrete length function as a non-discrete multiplicity symbol (see also Remark \ref{mult_rem}).

We remark that the first available examples of non-discrete length functions are those induced by a non-discrete rank one valuation $v : Q \to \R \cup \{ \infty \}$, where $Q$ is the field of quotients of a valuation domain $R$ and the value group $\Gamma(Q)$ is a dense additive subgroup of the reals (see \cite{NR}). These length functions and the induced $L$-entropies $\ent_L$ are our favourite sources of examples. 

\subsection*{Notation}
All along the paper, except when explicitly stated, $R$ denotes an arbitrary associative unitary ring (not necessarily commutative), and ${\rm Mod}(R)$ is the category of right $R$-modules. The ring of endomorphisms of a right $R$-module $M$ is denoted by $\End_R(M)$. $\mathcal F(M)$ denotes the family of all the finitely generated submodules of $M$. As usual, the ring of polynomials over $R$ is denoted by $R[X]$, and the ring of Laurent polynomials $R[X,X^{-1}]$ by $R[X^{\pm 1}]$. We denote by $\N, \N_{\geq1}, \R, \R_{\geq 0}$ the sets of the naturals, positive integers, reals and non-negative reals, respectively. The set $\R_{\geq 0} \cup \{\infty\}$ is denoted by $\R^*$.

\section{Extensions of scalars and length functions}\label{mod_pol}
In this section we explain the setting where entropy becomes a length function and the tool of the central localization; we refer to \cite{R} for this last notion.

\subsection{Length functions.} Length functions were axiomatized by Northcott and Reufel \cite{NR}, generalizing the composition length of modules. This concept was investigated by  V\'amos \cite{V1,V2}, Zanardo \cite{Z}, and, in the setting of Grothendieck categories, by the second named author \cite{VI}.

A function $L:{\rm Mod}(R)\to \R^*$ 
is an {\em invariant} if $L(0)=0$ and $L(M)=L(M')$ whenever $M\cong M'$. Furthermore, 
\begin{enumerate}[\rm (1)]
\item $L$ is {\em additive} if, for any short exact sequence $0\to M'\to M\to M''\to 0$ in ${\rm Mod}(R)$,
$L(M)=L(M')+L(M'')$;
\item $L$ is {\em upper continuous } if $L(M)=\sup\{L(F):F\in \F(M)\}$,
for all $M\in {\rm Mod}(R)$.
\end{enumerate}
An invariant which is both additive and upper continuous, is said to be a {\em length function}. A module $N$ such that $L(N)<\infty$ is said to have {\em finite $L$-length} or to be {\em $L$-finite}. Following \cite{SVV}, a module $M \in {\rm Mod}(R)$ is said to be {\em locally $L$-finite} if its cyclic (or, equivalently, finitely generated) submodules are $L$-finite. Given an $R$-module $M$, we adopt the following notations:
\begin{enumerate}[\rm --]
\item $\Fin(L)$ is the class of all $L$-finite modules, $\Fin_L(M)$ is the class of the submodules of $M$ of finite $L$-length;
\item $\lFin(L)$ is the class of locally $L$-finite modules. 
\end{enumerate}
Notice that $M$ is locally $L$-finite if and only if $\F (M) \subseteq \Fin_L(M)$. 

Some natural example of length functions are: the dimension of vector spaces, the logarithm of the cardinality of Abelian groups, the rank of modules over integral domains, the composition length of modules; notice that these examples also satisfy  the following property:

\begin{enumerate}[\rm (3)]
\item a  length function $L$ is {\em discrete} provided its set of finite values is a discrete (i.e., order isomorphic to $\N$) subset of $\R_{\geq 0}$. 
\end{enumerate}
An example of a non-discrete length function was found by Northcott and Reufel \cite{NR} (see also the last part of Section 4 and Example \ref{ex2.2}).


\begin{lemma} \label{fin_gen_chain}
Let $L: {\rm Mod}(R) \to \R^*$ be a length function. If $N\in \Fin(L)$, then there exists a countable ascending chain of finitely generated submodules of $N$:
$$0=N_0 \leq N_1 \leq \ldots \leq N_n \leq \ldots $$
such that $\sup_{n\in\N} L(N_n)=L(N)$.
\end{lemma}
\begin{proof}
Since $L(N)=\sup_{F \in \mathcal F (N)} L(F)$, for all $k\in\N_{\geq1}$ there exists $F_k\in\F(N)$ such that $L(F_k)>L(N)-1/k$. One can conclude letting $N_0=0$ and $N_n=F_1+\ldots+F_n$, so that $L(N_n)\geq L(F_n)>L(N)-1/n$, for all $n\in\N$.
\end{proof}

In the notation of the above lemma, if $L$ is discrete then the sequence $\{L(N_n)\}_{n\in\N}$ stabilizes so we can always find a finitely generated submodule $N'$ of $N$ such that $L(N')=L(N)$.

\begin{lemma} \label{prep}
Let $L: {\rm Mod}(R) \to \R^*$ be a length function, let $F\in\Fin(L)$, and let $F_0 \leq F_1 \leq \cdots \leq F_n \leq \cdots$ be an ascending sequence of submodules of $F$. If we let $F_\infty =\bigcup_{n\in\N} F_n$, then $\displaystyle\lim_{n \to \infty} L(F/F_n)= L(F/F_\infty)$.
\end{lemma}
\begin{proof}
Using the additivity and the upper continuity of $L$, we get:\\
$\displaystyle\lim_{n \to \infty} L({F}/{F_n})= \inf_{n\in\N} (L(F)-L(F_n))=L(F) - \sup_{n\in\N} L(F_n)=L(F) - L(F_\infty)\,.$ \end{proof}

\subsection{Central localization.}
Let us start fixing some conventions for right $R[X]$-modules. Indeed, we use the notation $M_\phi$, with $M$ a right $R$-module and $\phi\in\End_R(M)$, to denote the right $R[X]$-module $M_{R[X]}$, where $X$ acts on $M$ as $\phi$. An $R$-homomorphism $\alpha:M\to N$ induces an $R[X]$-homomorphism $M_\phi\to N_\psi$ if and only if $\alpha \cdot \phi = \psi \cdot \alpha$. For more details on these notions we refer to \cite{SVV}.

\smallskip
Notice that $\mod {R[X^{\pm 1}]}$ can be considered as the full subcategory of $\mod {R[X]}$ consisting of those objects $M_\phi$ for which $\phi$ is an automorphism. In fact, the inclusion $\mod {R[X^{\pm 1}]}\to \mod {R[X]}$ has a left adjoint functor 
$$-\otimes_{R[X]}R[X^{\pm1}]:\mod {R[X]}\to \mod {R[X^{\pm1}]}\,,$$ whose properties are crucial in the proof of AT:

\begin{lemma}\label{lemma_central_loc}
The functor $-\otimes_{R[X]}R[X^{\pm1}]$ is exact. Furthermore, given a right $R[X]$-module  $M_\phi$ and letting $\ker_\infty(\phi)=\bigcup_{n\in\N}\ker(\phi^n)$, the right $R[X]$-module  $M_\phi\otimes_{R[X]} R[X^{\pm1}]$ is isomorphic to a direct union of copies of $(M/\ker_\infty(\phi))_{\bar \phi}$, where $\bar \phi$ is the map induced by $\phi$. 
\end{lemma} 
\begin{proof}
This fact is explained in \cite[Section 1.10]{R}; in fact $R[X^{\pm1}]$ is the central localization of $R[X]$ with respect to the multiplicative system $\{X^n: n \geq 0\}$. Let us give a sketch of how to view $M_\phi\otimes_{R[X]} R[X^{\pm1}]$ as described above. By \cite[Proposition 1.10.18]{R}, $M_\phi\otimes_{R[X]} R[X^{\pm1}]$ is isomorphic to the direct limit of the following directed system: for all $n\in\N$ let $M_n=\hom_{R[X]}(X^nR[X],M_\phi)\cong M_\phi$; furthermore, given $n\leq m$, we let 
$$\phi_{m,n}:M_n\to M_{m} \ \ \ (\alpha :X^nR[X]\to M)\mapsto (\alpha\restriction_{X^mR[X]}:X^mR[X]\to M)\,.$$ 
This corresponds to the following directed system of right $R[X]$-modules:
$$M_\phi\overset{\phi}{\to}M_\phi\overset{\phi}{\to}\cdots\overset{\phi}{\to}M_\phi\to \cdots .$$
Of course, the direct limit of the above directed system is the union of the images of the canonical maps from each copy of $M_\phi$ to the direct limit. It is not difficult to check that the kernel of such maps is exactly $\ker_\infty(\phi)$.
\end{proof}

\section{Algebraic entropies induced by non-discrete length functions}\label{length_fun_subs}
In what follows $L:{\rm Mod}(R) \to \R^*$ always denotes an arbitrary length function. 
The {\em algebraic $L$-entropy} $\ent_L$, as defined in \cite{SZ} and \cite{SVV}, is an invariant of the category of right $R[X]$-modules:
$$\ent_L:{\rm Mod}{R[X]}\to \R^*.$$
Let us recall its definition for the reader's convenience. If $N \in \Fin_L(M)$ and $n \in \N_{\geq 1}$, $T_n(\phi, N)=\sum_{i=0}^{n-1} \phi^i N$ denotes the {\em $n$-th $\phi$-trajectory of $N$}, and $T(\phi, N)=\sum_{i=0}^\infty \phi^i N$ its {\em $\phi$-trajectory}. Given a right $R[X]$-module $M_\phi$, we let
$$\ent_L(M_\phi)=\ent_L(\phi)= \sup \{ \ent_L(\phi, N) : N \in \Fin_L(M) \}\,,$$
where
\begin{equation}\label{eq_def_ent}\ent_L(\phi, N)=\lim_{n \to \infty} {L(T_n(\phi,N)}/{n}\,.\end{equation}
By the additivity of $L$, the sequence of real numbers $\{ L(T_n(\phi,N))\}_{n\in\N}$ is sub-additive and non-negative. Then, by the well-known Fekete Lemma (see \cite{F}), the limit \eqref{eq_def_ent} exists finite and it coincides with 
$$ \inf \left\{ {L(T_n(\phi,N)})/{n}:n\in\N \right\}\,.$$
In what follows we give an alternative formula to compute entropy which will be very useful in the rest of the paper. Let us first recall the following result from \cite{SZ}:

\begin{lemma}\label{SZ1} {\rm (Lemma 1.9 in \cite{SZ})}.
Let $M_\phi$ be a right $R[X]$-module and let $N \in \Fin_L(M)$. For all $n \geq 1$, let 
$$\alpha_n=L\left(\frac{T_{n+1}(\phi,N)}{T_{n}(\phi,N)}\right)\,.$$ 
Then the sequence
of non-negative real numbers $\{ \alpha_n \}_n$ is non-increasing, hence 
$\lim_{n\to \infty} \alpha_n = \inf_n \alpha_n$.
\end{lemma}

In case $L$ is a discrete length function, from the preceding lemma one easily deduces that $\ent_L(\phi,N)= \inf_n \alpha_n= \alpha$, where $\alpha$ is the eventually constant value of the stationary sequence $\{ \alpha_n \}_n$ (see Proposition 1.10 in \cite{SZ}). We now extend this result to arbitrary length functions:

\begin{proposition}\label{non-discrete}
In the setting of Lemma \ref{SZ1}, $\ent_L(\phi,N)= \inf_n \alpha_n\ .$
\end{proposition}
\begin{proof}
Let $\alpha= \inf_{n} \alpha_n$. For any $\varepsilon >0$ there exists an index $n_0$ such that 
$$L\left(\frac{T_{n+1}(\phi,N)}{T_{n}(\phi,N)}\right) < \alpha + \varepsilon\,, $$ 
for all $n \geq n_0$. By the additivity of $L$, one shows inductively that, for all $k\geq 1$,
$\displaystyle L(T_{n_0+k}(\phi,N))= L(T_{n_0}(\phi,N))+\sum_{i=0}^{k-1}L\left(\frac{T_{n_0+i+1}(\phi,N)}{T_{n_0+i}(\phi,N)}\right)$.
Thus,
$$L(T_{n_0}(\phi,N))+k \alpha\leq L(T_{n_0+k}(\phi,N))\leq L(T_{n_0}(\phi,N))+k (\alpha+ \epsilon)\,.$$
Using the above inequalities we obtain: 
$$\lim_{k \to \infty} \frac{L(T_{n_0+k}(\phi,N))}{n_0+k} \leq \lim_{k \to \infty} \frac{L(T_{n_0}(\phi,N))+k (\alpha+ \epsilon)}{n_0+k}= \alpha + \epsilon\,,$$
and
$$\lim_{k \to \infty} \frac{L(T_{n_0+k}(\phi,N))}{n_0+k} \geq \lim_{k \to \infty} \frac{L(T_{n_0}(\phi,N))+k \alpha}{n_0+k}= \alpha\ .$$
From these inequalities we deduce that $\ent_L(\phi,N)= \alpha $. \end{proof}

\begin{corollary}\label{coro_comp_with_traj}\label{entropia_con_potenza}
In the setting of Lemma \ref{SZ1} we have: 
\begin{enumerate}[\rm (1)]
\item $\ent_L(\phi,N)=\ent_L(\phi,T_n(\phi,N))$, for all $n \geq 1$;
\item $\ent_L(\phi,N)=\ent_L(\phi, \phi^n(N))$, for all $n\in\N$.
\end{enumerate}
\end{corollary}
\begin{proof}
Part (1) is an easy consequence of Proposition \ref{non-discrete}. For part (2) just notice that $\ent_L(\phi,\phi^n(N))\leq \ent_L(\phi,T_n(\phi,N))=\ent_L(\phi,N)$. For the converse inequality we have, for any fixed $n \in \N$: 
\begin{align*}
\ent_L(\phi,F)&=\lim_{m\to\infty}\frac{L(T_m(\phi,F))}{m} \leq \lim_{n<m\to\infty}\frac{L(T_n(\phi,F))+L(T_{m-n}(\phi, \phi^n(F)))}{m}\\
&=\lim_{n<m\to\infty}\frac{L(T_n(\phi,F))}{m}+\frac{L(T_{m-n}(\phi, \phi^n(F)))}{m-n}\cdot\frac{m-n}{m}\\
&=\lim_{m\to\infty}\frac{L(T_{m}(\phi, \phi^n(F)))}{m}=\ent_L(\phi,\phi^n(F))\,.\qedhere\end{align*}
\end{proof}

Let us conclude this section with an example of computation of entropy for modules over (not necessarily commutative) domains, which shows that AT does not hold in general for $\ent_L$ on the whole category $\mod R$.

\begin{example}\label{nolocally}
Let $R$ be a domain and let $L: {\rm Mod }(R) \to \R^*$ be a non-trivial length function (i.e., $L$ takes some finite non-zero value). Assume that $\lFin(L)\neq \mod R$ (this happens for example if $L=\ell$ is the composition length for $R$ not Artinian, or $L=\log|-|$ in $\mod \Z$). In what follows we exhibit a short exact sequence
$$0\to N_{\phi\restriction_N}\to M_\phi \to (M/N)_{\bar \phi}\to 0$$
such that $\ent_L(\phi\restriction_N)=\ent_L(\phi)=0$ and $\ent_L(\bar \phi)>0$. Notice first that $L(R)=\infty$; in fact, otherwise all the finitely generated modules would have finite $L$-length, contradicting the inequality $\lFin(L)\neq \mod R$. Furthermore, since $R$ is a domain, any $R$-submodule $F\leq R^{(\N)}$ contains a copy of $R$, so that $L(F)=\infty$. 
Being $L$ non-trivial, one can find a right ideal $K\leq R$ such that $0<L(R/K)<\infty$.
Let $M_\phi=R\otimes_R R[X]=R[X]$ and notice that, since $\Fin_L(M)=0$, $\ent_L(\phi)=0$. Similarly, if $N_{\phi\restriction_N}=K\otimes_R R[X]$, then $\ent_L(\phi \restriction_N)=0$. Now $(M/N)_{\bar\phi}$ is isomorphic to $(R/K)\otimes_R R[X]$ (since $-\otimes_R R[X]$ is exact) and, proceeding as in \cite[Example 2.14]{SVV}, and using Lemma \ref{SZ1}, one can see that $\ent_L(\bar \phi)=L(R/K)>0$.

\end{example}

\section{The entropy $\ent_L$ is upper continuous}\label{upp_cont_sec}

Our goal in this subsection is to extend to possibly non-discrete length functions Proposition 2.12 in \cite{SVV}.  Before that, we need the following technical result.

\begin{lemma} \label{prop_base_dense}
Let $M_\phi$ be a right $R[X]$-module, let $\varepsilon$ be a positive real number, and let $N_0 \leq N \in\Fin_L(M)$ be such that $L(N)-L(N_0)<\varepsilon$. Then,
\begin{enumerate}[\rm (1)]
\item $L(\phi (N))-L(\phi (N_0))<\varepsilon$;
\item $L(T_n(\phi, N))-L(T_n(\phi ,N_0))<n\varepsilon$, for all $n \geq 1$;
\item $\ent_L(\phi,N)-\ent_L(\phi,N_0)\leq \varepsilon$.
\end{enumerate}
\end{lemma}

\begin{proof}
(1) Notice that there is an epimorphism $N/N_0\to \phi(N)/\phi(N_0)$. 
Then, 
$$L(\phi(N))-L(\phi(N_0))=L(\phi(N)/\phi(N_0))\leq L(N/N_0)=L(N)-L(N_0)<\varepsilon\,.$$

\smallskip\noindent
(2) 
For $n=1$ the claim is part of our hypotheses. Assume now that the claim holds for some $n$ and let us verify it for $n+1$. Consider the following exact sequences
$$\frac{T_{n}(\phi,N)}{T_{n}(\phi,N_0)}\to\frac{T_{n+1}(\phi,N)}{T_{n+1}(\phi,N_0)+\phi^n(N)}\to 0\,,$$
$$ \frac{\phi^n(N)}{\phi^n(N_0)}\to \frac{\phi^n(N)}{\phi^n(N)\cap T_{n+1}(\phi,N_0)}\to 0\,.$$
Using the inductive hypothesis and the first sequence above one can see that 
$$L\left(\frac{T_{n+1}(\phi,N)}{T_{n+1}(\phi,N_0)+\phi^n(N)}\right)<n\varepsilon\,,$$
while using part (1) and the second exact sequence one obtains that
$$L\left(\frac{\phi^n(N)}{\phi^n(N)\cap T_{n+1}(\phi,N_0)}\right)<\varepsilon\,.$$
Consider now the following short exact sequence:
$$0\to \frac{\phi^n(N)}{\phi^n(N)\cap T_{n+1}(\phi,N_0)}\to \frac{T_{n+1}(\phi,N)}{T_{n+1}(\phi,N_0)}\to \frac{T_{n+1}(\phi,N)}{T_{n+1}(\phi,N_0)+\phi^n(N)}\to0\ .$$
By the additivity of $L$ and the above computations we obtain  
$$L\left(\frac{T_{n+1}(\phi,N)}{T_{n+1}(\phi,N_0)}\right)<(n+1)\varepsilon\  ,$$
that is, $L(T_{n+1}(\phi,N))-L(T_{n+1}(\phi,N_0))<(n+1)\varepsilon$.

\smallskip\noindent
(3) follows by (2) just dividing by $n$ and passing to the limit.
\end{proof}

We can now prove the main result of this section.

\begin{proposition}\label{sanatoria}
Let $M_\phi$ be a right $R[X]$-module. Then,
$$\ent_L(\phi)= \sup \{ \ent_L(\phi, F) : F \in \mathcal F (M)\cap \Fin_L(M) \}.$$
\end{proposition}
\begin{proof}
Let $N\in\Fin_L(M)$. By Lemma \ref{fin_gen_chain}, there is an ascending chain $0=N_0 \leq N_1 \leq \ldots \leq N_k \leq \ldots $ of finitely generated submodules of $N$ such that $L(N)=\sup_{k\in\N}L(N_k)$. For all $\varepsilon>0$ there exists $k\in\N$ such that $L(N)-L(N_k)<\varepsilon$ and so, by Lemma \ref{prop_base_dense}, $\ent_L(\phi,N)-\ent_L(\phi,N_k)\leq \varepsilon$. This shows that 
$\ent_L(\phi,N)=\sup_{k\in \N}\ent_L(\phi,N_k)$, and the conclusion immediately follows.
\end{proof}

As an immediate consequence we derive the following corollary (for a proof of the second part, see \cite[Proposition 8]{V1}):

\begin{corollary} \label{upp_cont}
The algebraic entropy $\ent_L$ is an upper continuous invariant of $\mod {R[X]}$. In particular, if a right $R[X]$-module $M_\phi$ is the direct union of a directed system of $R[X]$-submodules $\{(N_\alpha)_{\phi_\alpha}: \alpha\in \Lambda\}$, then 
$$\ent_L(M_\phi)=\sup_{\alpha\in\Lambda}\ent_L((N_\alpha)_{\phi_\alpha})\,.$$
\end{corollary}

Another consequence of Proposition \ref{sanatoria} is the following:

\begin{corollary}\label{gen_da_fin}
Let $M_\phi$ be a right $R[X]$-module and suppose that $M=T(\phi,N)$ for some $N\in\Fin_L(M)$. Then, $\ent_L(\phi)=\ent_L(\phi,N)$.
\end{corollary}
\begin{proof}
By hypothesis, $M=\bigcup_{n \geq 1}T_n(\phi,N)$ and so, for all $F\in \F(M)\cap \Fin_L(M)$ there exists $n\geq 1$ such that $F$ is contained into $T_n(\phi,N)$. Hence $\ent_L(\phi,F)\leq \ent_L(\phi,T_n(\phi,N))=\ent_L(\phi,N)$, by Corollary \ref{coro_comp_with_traj}, and the claim follows.
\end{proof}

We illustrate now the main source of examples of non-discrete length functions $L$ and of the induced entropies $\ent_L$. For this purpose, we resume Example 2.2 of \cite{SZ}, slightly changing the notation.  Let $R$ be a non-discrete archimedean (i.e., rank one) valuation domain, with value group $\Gamma(R)$ isomorphic to a dense subgroup of $\R$. Let $v: Q \to \R \cup \{ \infty \}$ denote the valuation on the field of quotients $Q$ of $R$. According to \cite{NR}, a non-discrete length function $L_v: {\rm Mod} (R) \to \R^*$ is defined and uniquely identified by setting $L_v (R/I)= \inf \{ v(a) : a \in I \}$ for $I$ a non-zero ideal of $R$, and  $L_v (R)= \infty$. Since any positive real number $r$ can be reached as $\inf \{ v(a) : a \in I \}$ for a suitable ideal $I$ of $R$, the length function $L_v$ is non-discrete. We don't care whether $R$ is almost maximal, an assumption made in Example 2.2 of \cite{SZ} which is not relevant for our purposes. Notice that, if $P$ is the maximal ideal of $R$, then $L_v(R/P)=0$, hence all the semi-Artinian $R$-modules have zero $L_v$-length. 

Using the above notation, we take up again, in the next Example \ref{ex2.2}, the last part of Example 2.2 of \cite{SZ}, where the entropy $\ent_{L_v}(\phi)$ was computed using the formula $\ent_{L_v}(\phi)= \sup \{ \ent_{L_v}(\phi, F) : F \in \mathcal F (M) \}$; note that this formula is validated by Proposition \ref{sanatoria}. Example \ref{ex2.2} provides a negative answer to the following question: 
\begin{quotation}
Given a length function $L$, is it true that, if an endomorphism $\phi: M \to M$ satisfies $\ent_L(\phi)=0$, then all the cyclic trajectories $T(\phi, Rx)$ ($x \in M$) have finite $L$-length?
\end{quotation}
The answer is positive if $L$ is discrete, since $L(T(\phi, Rx))=L(T_n(\phi, Rx))$ for some integer $n$. Note that the converse implication is always true. 

\begin{example}\label{ex2.2}
Let $0=I_0 < I_1 < I_2 < \ldots < I_n < \ldots$ be an ascending sequence of ideals of the non-discrete archimedean valuation domain $R$ such that  $L_v(R/I_n)=\sqrt n - \sqrt {n-1}$ for each $n \geq 1$. Then $ \sum_{1 \leq j \leq n} L_v(R/I_j)= \sqrt n$ for each $n$. Let $M=\bigoplus_{n\geq 1} x_nR$, where $x_nR \cong R/I_n$ for all $n\geq 1$. In Example 2.2 of \cite{SZ} it was shown that, for every endomorphism $\phi$ of $M$, $\ent_{L_v}(\phi)=0$. Now, if $\phi$ is the endomorphism defined by setting $\phi(x_n)=x_{n+1}$ for all $n$, the trajectory $T(\phi,x_1R)$ coincides with $M$, and $L_v(M)= \sup_n \sum_{j \leq n} L_v(R/I_j)= \sup_n \sqrt n = \infty.$
\end{example}

\section{The Addition Theorem for $\ent_L$}
Let $L: {\rm Mod}(R) \to \R^*$ be a length function. In what follows we denote by $\lFin_L[X]$  and $\lFin_L[X^{\pm1}]$ the full subcategories of ${\rm Mod}(R[X])$ and ${\rm Mod}(R[X^{\pm1}])$, respectively, consisting of the modules $M_\phi$ such that $M \in \lFin(L)$. 
In this section we prove that the restriction of $\ent_L$ to $\lFin_L[X]$ is additive and so, by Corollary \ref{upp_cont}, it is a length function.  Our proof splits into two parts. In Subsection 5.1 we prove the additivity of $\ent_L$ on $\lFin_L[X^{\pm1}]$, and then, in Subsection 5.2, we show that the additivity on $\lFin_L[X^{\pm1}]$ implies additivity on $\lFin_L[X]$.

\subsection{Additivity in $\mod {R[X^{\pm 1}]}$}\label{inv_add_subs}

If $\phi$ is an automorphism, we have the nice formula displayed in the next Proposition \ref{nolimit}, which allows us to escape the limit calculation in the computation of $\ent_L$, and to give a direct proof of AT for locally $L$-finite modules in a general setting. This formula extends a similar formula for Abelian groups (see \cite{DG1}) and its counterpart for the intrinsic algebraic entropy proved in \cite{GV}.

\begin{lemma}\label{nolimit_lemma}
Let $M_\phi$ be a right $R[X^{\pm 1}]$-module and let $F\in \Fin_L(M)$. Then,
$$\ent_L(\phi,F)=L\left(\frac{T(\phi^{-1},F)}{\phi^{-1}T(\phi^{-1},F)}\right)\,.$$
\end{lemma}
\begin{proof}
For each $n \geq 1$ let $F_n=F \cap \phi^{-1}(T_n(\phi^{-1},F))$, and $F_\infty=\bigcup_{n}F_n=F \cap \phi^{-1}(T(\phi^{-1},F)$. It is not difficult to see that  
$$\frac{F}{F_n} \cong \frac{T_{n+1}(\phi^{-1},F)}{\phi^{-1}(T_{n}(\phi^{-1},F))} \ \ \ \text{and}\ \ \ \frac{F}{F_\infty} \cong \frac{T(\phi^{-1},F)}{\phi^{-1}(T(\phi^{-1},F))}\,.$$
The claim now follows from the following series of equalities:
\begin{align*}
\ent_L(\phi,F)&\overset{(*)}{=}\lim_{n \to \infty} L\left(\frac{T_{n+1}(\phi,F)}{T_{n}(\phi,F)}\right)=\lim_{n \to \infty} L\left(\frac{\phi^{-n}(T_{n+1}(\phi,F))}{\phi^{-n}(T_{n}(\phi,F))}\right)\\
&=\lim_{n \to \infty} L\left(\frac{T_{n+1}(\phi^{-1},F)}{\phi^{-1}(T_{n}(\phi^{-1},F))}\right)\\
&=\lim_{n \to \infty} L(F/F_n)\overset{(**)}{=}L(F/F_{\infty})=L\left(\frac{T(\phi^{-1},F)}{\phi^{-1}(T(\phi^{-1},F))}\right)\,,
\end{align*}
where $(*)$ holds by Lemma \ref{SZ1} and $(**)$ uses Lemma \ref{prep}. 
\end{proof}

\begin{proposition} \label{nolimit}
Let $M_\phi\in \lFin_L[X^{\pm 1}]$. Then
$$\ent_L(\phi)=\sup\{L(N/\phi^{-1}N) : N=T(\phi^{-1},N),\ L(N/\phi^{-1}N)<\infty\}\,.$$
\end{proposition}
\begin{proof}
The inequality $\leq$ follows by Lemma \ref{nolimit_lemma}. On the other hand, let $N\leq M$ be such that $N=T(\phi^{-1},N)$ (equivalently, $\phi^{-1}N \leq N$)  and $L(N/\phi^{-1}N)<\infty$. By Lemma \ref{fin_gen_chain}, we can find a sequence $\bar F_1\subseteq \bar F_2\subseteq \ldots \subseteq \bar F_n\subseteq \ldots$ of finitely generated submodules of $N/\phi^{-1}N$ such that, for all $n\geq 1$:
$$L(N/\phi^{-1}N)-L(\bar F_n)<{1}/{n}\,.$$
For all $n \geq1$ choose a finitely generated (hence in $\Fin(L)$) submodule $F_n$ of $N$ such that $(F_n+\phi^{-1}N)/\phi^{-1}N=\bar F_n$. By Lemma \ref{nolimit_lemma} and using the inclusion $T(\phi^{-1},F_n)\subseteq N$, we get for all $n \geq 1$
\begin{align*}
\ent_L(\phi,F_n)&=L\left(\frac{T(\phi^{-1},F_n)}{\phi^{-1}(T(\phi^{-1},F_n))}\right)\geq L\left(\frac{F_n}{F_n\cap \phi^{-1}N}\right)>L\left(\frac{N}{\phi^{-1}N}\right)-\frac{1}{n}\,.
\end{align*}
Taking the supremum for $n\geq 1$ we obtain $\ent_L(\phi)\geq L(N/\phi^{-1}N)$.
\end{proof}

We are now ready for the proof of the main result of this subsection.

\begin{theorem}\label{AT_for_inv}
Let $0\to N_{\phi\restriction_N}\to M_\phi \to (M/N)_{\bar\phi}\to 0$ be a short exact sequence in $\lFin_L[X^{\pm 1}]$. Then, $\ent_L(\phi)=\ent_L(\phi\restriction_N)+\ent_L(\bar \phi)$.
\end{theorem}
\begin{proof}
Let $F$ be a $\phi^{-1}$-invariant submodule of $M$ such that $L(F/\phi^{-1}F)<\infty$. We set $F'=F\cap N$ and $\bar F=(F+N)/N$; clearly $F'$ and $\bar F$ are $\phi^{-1}$-invariant and $\bar \phi^{-1}$-invariant, respectively. Using the equality $N=\phi^{-1}N$ and the fact that $\phi^{-1}$ commutes with intersection of submodules, we get the following isomorphisms:
$$\frac{ F'}{\phi^{-1}F'} \cong \frac{F \cap(\phi^{-1}F+N)}{\phi^{-1}F} \ \ , \ \ \frac{F}{F \cap(\phi^{-1}F+N)} \cong \frac{ \bar F}{\bar \phi^{-1} \bar F} \ .$$
From these isomorphisms we obtain the exact sequence
\begin{equation}\label{sostituto_snake} 
0\to  \frac{ F'}{\phi^{-1}F'}\to \frac{F}{\phi^{-1}F}\to \frac{ \bar F}{\bar \phi^{-1} \bar F}\to 0
\end{equation}
The short exact sequence \eqref{sostituto_snake}, together with Proposition \ref{nolimit}, shows that
$$L(F/\phi^{-1}F)=L(F'/\phi^{-1}F')+L(\bar F/\bar \phi^{-1}\bar F)\leq \ent_{L}(\phi')+\ent_L(\bar \phi)\,.$$
By the arbitrariness of $F$, we obtain that $\ent_L(\phi)\leq \ent_{L}(\phi')+\ent_L(\bar \phi)$.

For the converse inequality choose $F_1\in \F(M)$, $F_2\in \F(N)$, and let $F=F_1+F_2$. Then $\bar F=(F+N)/N=(F_1+N)/N$ and $F_2\subseteq F\cap N$; thus,
\begin{align*}
\ent_L(\phi,F)&=\lim_{n\to\infty}\frac{L(T_n(\phi,F))}{n}\\
&= \lim_{n\to\infty}\frac{L(T_n(\phi,F_2))+L(T_n(\phi,F)/T_n(\phi,F_2))}{n}\\
&\geq \lim_{n\to\infty}\frac{L(T_n(\phi,F_2))+L(T_n(\phi,F)/(N\cap T_n(\phi,F)))}{n}\\
&=\lim_{n\to\infty}\frac{L(T_n(\phi,F_2))+L(T_n(\bar \phi,\bar F))}{n}\\
&=\ent_L(\phi,F_2)+\ent_L(\bar \phi, \bar F)\,.
\end{align*}
One concludes by the arbitrariness of $F_1$ and $F_2$. \end{proof}

\begin{remark}\label{mult_rem}
Let $R$ be a right Noetherian ring and let $L:\mod R\to \R^*$ be a length function. Peter V\'amos \cite{V1} generalized the classical notion of multiplicity symbol (see for example \cite[Chapter 7]{N}), defining an $L$-multiplicity $\mult_L$, that is a length function of the category $\mod {R[X]}$ attached to $L$. The classical multiplicity symbol is recovered by taking $L=\ell$ to be the composition length.\\
Let us recall the definition of $\mult_L$: for a finitely generated $R[X]$-module $N_\phi$ 
$$\mult_L(N_\phi)=\begin{cases}
L(N/\phi N)-L(\ker(\phi))&\text{if $L(N/\phi N)<\infty$;}\\
\infty&\text{otherwise.}
\end{cases}$$
For an arbitrary $R[X]$-module $M_\phi$, $\mult_L(M_\phi)=\sup\{\mult_L(N_{\phi\restriction_N}):N_{\phi\restriction_N} \text{ f.g.}\}$.
Notice in particular that, if $\phi$ is bijective, then $\ker(\phi)=0$ and we obtain
$$\mult_L(M_\phi)=\sup\{L(N/\phi N):N_{\phi\restriction_N} \text{ f.g.}\}\,.$$
When $\phi$ an automorphism, one can also consider the $R[X]$-module $M_{\phi^{-1}}$ and, by Proposition \ref{sanatoria} and Lemma \ref{nolimit_lemma}, we obtain
$$\ent_L(M_{\phi^{-1}})=\sup\{L(N/\phi N):N_{\phi\restriction_N} \text{ f.g.}\}\,.$$
In particular, $\mult_L(M_\phi)=\ent_L(M_{\phi^{-1}})$ for any $R[X^{\pm1}]$-module $M_\phi$.
\end{remark}

\subsection{Reduction to $R[X^{\pm 1}]$-modules}\label{red_subs}
The key result of this subsection is the following

\begin{proposition}\label{hyperkernel} 
Let $M_\phi\in\lFin_L[X]$, let $K= \bigcup_{n\in\N} \Ker( \phi^n)$ and denote by $\bar\phi: M/K \to M/K$ the (injective) map induced by $\phi$. Then, $\ent_L(\phi)=\ent_L(\bar \phi)$.
\end{proposition}

We will prove the above proposition at the end of this subsection, but let us first show how it leads to the proof of AT. First we deduce the following

\begin{corollary}\label{coro_loc_AT}
Let $M_\phi\in\lFin_L[X]$ and let $N_\psi=M_\phi\otimes_{R[X]}R[X^{\pm1}]$. Then, $\ent_L(\phi)=\ent_L(\psi)$.
\end{corollary}
\begin{proof}
By Lemma \ref{lemma_central_loc}, $N=\bigcup_{\alpha\in\Lambda} N_\alpha$, where $\{N_\alpha:\alpha\in\Lambda\}$ is a directed system of $\psi$-invariant submodules such that $(N_\alpha)_{\psi\restriction_{N_\alpha}}\cong (M/\ker_\infty(\phi))_{\bar\phi}$ for all $\alpha\in\Lambda$.  By Proposition \ref{hyperkernel}, $\ent_L(\psi\restriction_{N_\alpha})=\ent_L(\bar \phi)$ for all $\alpha\in\Lambda$. One can conclude using Corollary \ref{upp_cont}; in fact, $\ent_L(\psi)=\sup_{\alpha\in\Lambda}\ent_L(\psi\restriction_{N_\alpha})=\ent_L(\phi)$.
\end{proof}


\begin{theorem}[Addition Theorem]\label{coro_AT}
The restriction $\ent_L:\lFin_L[X]\to \R^*$ is additive, hence it is a length function.
\end{theorem}
\begin{proof}
Let $N_{\phi'}\leq M_\phi \in \lFin_L[X]$ and denote their quotient by $(M/N)_{\bar \phi}$. Then,
\begin{align*}
\ent_L(M_\phi)&\overset{(*)}{=}\ent_L(M_\phi \otimes_{R[X]}R[X^{\pm1}])\\
&\overset{(**)}{=}\ent_L(N_{\phi'}\otimes_{R[X]}R[X^{\pm1}])+\ent_L((M/N)_{\bar\phi}\otimes_{R[X]}R[X^{\pm1}])\\
&\overset{(*)}{=}\ent_L(N_{\phi'})+\ent_L((M/N)_{\bar \phi})
\end{align*}
where the equalities $(*)$ are an application of Corollary \ref{coro_loc_AT}, while the equality $(**)$ comes from Theorem \ref{AT_for_inv} and the exactness of the functor $-\otimes_{R[X]}R[X^{\pm 1}]$, proved in Lemma \ref{lemma_central_loc}. The final claim follows by Corollary \ref{upp_cont}.
\end{proof}

The rest of this subsection is devoted to the proof of Proposition \ref{hyperkernel}; in what follows we keep the setting and notation of the statement of that proposition. Let us first prove the following technical lemma.

\begin{lemma} \label{1}
Let $F\in\Fin_L(M)$ and let $\bar F=F\cap K$. Then
\begin{enumerate}[\rm (1)]
\item $\phi^n(F\cap K)=\phi^n F \cap K$;
\item there is an increasing sequence $(k_n)_{n\in\N}$ of positive integers
 such that
$$L(\phi^{k_n} F )-L(\bar \phi^{k_n} \bar F)<1/n\,.$$
\end{enumerate}
\end{lemma}

\begin{proof}
(1) If $x\in \phi^n F$ then $x=\phi^n(y)$ for some $y \in F$, while if $x\in K$ then $\phi^m(x)=0$ for some $m\in\N$. Thus, if $x\in \phi^n F\cap K$, then there exist $y\in F$ and $m\in\N$ such that $0=\phi^m(x)=\phi^{n+m}(y)$, which shows that $y\in K$ and so $x\in \phi^{n}(F\cap K)$. The other inclusion is obvious.

\smallskip\noindent
(2) Since $F\in \Fin_L(M)$,  $F\cap K$ is $L$-finite. By Lemma \ref{fin_gen_chain}, there is a sequence $0=F_0\leq F_1\leq \ldots \leq F_n\leq \ldots $ of finitely generated submodules of $F\cap K$ such that $L(F\cap K)-L(F_n)<1/n$. Since each $F_n$ is finitely generated, there exists an increasing sequence of positive integers $(k_n)_{n\in\N}$ such that $\phi ^{k_n}(F_n)=0$. Consider, for each $n$, the short exact sequence:
$$0\to \phi^{k_n} F \cap K \to \phi^{k_n} F \to \bar\phi^{k_n} \bar F \to 0 \ .$$
By part (1),  $\phi^{k_n} F \cap K= \phi^{k_n} (F \cap K)$, and since this module is a quotient of 
$(F\cap K)/F_n$, then $L(\phi^{k_n} F )-L(\bar \phi^{k_n} \bar F)=L(\phi^{k_n} F \cap K)<1/n$ by construction.  
\end{proof}

\begin{proof}[Proof of Proposition \ref{hyperkernel}]
Let $F \in \Fin_L(M)$. We will prove that $\ent_L(\phi,F)=\ent_L(\bar \phi,\bar F)$, where $\bar F=(F+K)/K$.
Using Corollary \ref{entropia_con_potenza} both for $M$ and $M/K$, it is enough to check that, for every $n$, there exists $k_n$ such that
the following inequality holds: $\ent_L(\phi, \phi^{k_n}F) - \ent_L(\bar\phi,  \bar\phi^{k_n} \bar F) \leq 1/n$.
We can choose a sequence $(k_n)_{n\in\N}$ as in Lemma \ref{1} (2); then it is easily seen (using an argument similar to that used in Lemma \ref{prop_base_dense}) that, for each $m>1$,
$$L(T_m(\phi, \phi^{k_n}F)) - L(T_m( \bar\phi, \bar\phi^{k_n}\bar F)) < m/n\,.$$
Dividing by $m$ and taking the limit for $m \to \infty$ we get the desired inequality. 
\end{proof}

\section{Entropy in many variables}

Let $k$ be a positive integer; in this section we extend the definition of entropy to modules over the ring of polynomials $R[X_1,\dots,X_k]$, that is, we define an invariant
$$\ent_L:\mod {R[X_1,\dots,X_k]}\to \R^*\,.$$
As we did for $\mod {R[X]}$, we consider a right  ${R[X_1,\dots,X_k]}$-module $M_{R[X_1,\dots,X_k]}$ as a pair $M_{\Phi}$, where $\Phi=(\phi_1,\dots,\phi_k)$ is a $k$-tuple of pair-wise commuting endomorphisms of $M_R$, with $\phi_i$ representing the action of $X_i$ on $M$. For any $L$-finite submodule $N_R\leq M$, the {\em $n$-th $\Phi$-trajectory of $N$} is
$$T_n(\Phi,N)=\sum_{i_1,\dots,i_k=0}^{n-1}\phi_1^{i_1}\ldots\phi_k^{i_k}N\,.$$ 
The {\em $L$-entropy} of $M_\Phi$ is defined as 
$$\ent_L(M_\Phi)=\ent_L(\Phi)= \sup \{ \ent_L(\Phi, N) : N \in \Fin_L(M) \}\,,$$
where
\begin{equation}\label{eq_def_poli_ent}
\ent_L(\Phi, N)=\lim_{n \to \infty} \frac{1}{n^k}L(T_n(\Phi,N))\,.
\end{equation}
The existence of the above limit follows by (a weak form of) a result of Ceccherini-Silberstein, Coornaert and Krieger \cite{amenable_monoid} that generalizes Fekete's Lemma to sub-additive functions on cancellative amenable semigroups. More explicitly,  the monoid $\N^k$ is clearly cancellative and one can check its amenability (there is no need to distinguish left and right  here as $\N^k$ is commutative) using the F\o lner sequence $\{F_n\}_{n\in\N_{\geq 1}}$, where
$$F_n=\{(h_1,\dots,h_n)\in\N^k:h_i< n\}\,,$$
for all $n\in\N_{\geq 1}$. Denote by $\mathcal P^{fin}(\N^k)$ the set of finite parts of $\N^k$, let $N\in\Fin_L(M)$ and denote by $\rho:\N^k\to \End(M)$ the monoid homomorphism mapping $e_i\mapsto \phi_i$, where $\{e_i:i=1,\dots,k\}$ is the standard basis of $\N^k$. By the additivity of $L$, the function
$$\textstyle f_N:\mathcal P^{fin}(\N^k)\to \R_{\geq 0} \ \ \ f_N(F)=L\left (\sum_{f\in F} \rho_f(N)\right)$$
satisfies the assumptions of \cite[Theorem 1.1]{amenable_monoid}. The conclusion of \cite[Theorem 1.1]{amenable_monoid} tells us that $\lim_{n\to\infty} f_N(F_n)/|F_n|$ exists finite, that is, the limit in \eqref{eq_def_poli_ent} exists finite.

\smallskip
Let $\lFin_L[X_1,\dots,X_k]$ be the class of all the right $R[X_1,\dots,X_k]$-modules $M_{\Phi}$ such that $M_R\in\lFin(L)$. The main result of this section is the following

\begin{theorem}
The restriction $\ent_L:\lFin_L[X_1,\dots,X_k]\to \R^*$ is a length function.
\end{theorem}
\begin{proof}
We proceed by induction on $k\geq 1$. The case $k=1$ is Theorem \ref{coro_AT}. Suppose now that $k>1$ and
$\ent_L:\lFin_L[X_1,\dots,X_{k-1}]\to \R^*$ is a length function. We fix the following notation:
\begin{enumerate}[\rm --]
\item $\C_{k-1}=\lFin_L[X_1,\dots,X_{k-1}]$, so that  $\C_{k-1}[X_k]=\lFin_L[X_1,\dots,X_{k}]$;
\item $L_{k-1}=\ent_L:\C_{k-1}\to \R^*$ and $L_{k}=\ent_L:\C_{k-1}[X_k]\to \R^*$.
\end{enumerate}
Let $(M_{\phi_1,\dots,\phi_{k-1}})_{\phi_k}\in \C_{k-1}[X_k]$ (so $M_{\phi_1,\dots, \phi_{k-1}}\in\C_{k-1}$) and let $N\in \Fin_L(M)$; then 
\begin{align*}
\ent_L({\phi_1,\dots,\phi_{k}},N)&=\lim_{n\to\infty}\frac{L(T_n({\phi_1,\dots,\phi_{k}},N))}{n^k}\\
&=\lim_{n\to\infty}\frac{L(T_n(\phi_1,\dots,\phi_{k-1},T_n(\phi_k,N)))}{n^k}\\
&=\lim_{m\to\infty}\frac{1}{m} \lim_{n\to\infty}\frac{L(T_n(\phi_1,\dots,\phi_{k-1},T_m(\phi_k,N)))}{n^{k-1}}\\
&\overset{(*)}{=}\lim_{m\to\infty}\frac{L_{k-1}(T_m(\phi_k, N)R[X_1,\dots,X_{k-1}])}{m}\\
&{=}\lim_{m\to\infty}\frac{L_{k-1}(T_m(\phi_k, NR[X_1,\dots,X_{k-1}]))}{m}\\
&=\ent_{L_{k-1}}(\phi_k, NR[X_1,\dots,X_{k-1}])\,,
\end{align*} 
where equality $(*)$ follows by Corollary \ref{gen_da_fin}. Taking the supremum with $N$ ranging in $\Fin_L(M)$ we obtain that $L_{k}=\ent_{L_{k-1}}$ and we already know that $\ent_{L_{k-1}}$ is a length function of $\C_{k-1}[X_k]$ by Theorem \ref{coro_AT}.
\end{proof}

\section{The Uniqueness Theorem for $\ent_L$}\label{Unique_section}

One application of the Addition Theorem in \cite{SVV} is the Uniqueness Theorem (UT, for short). It states that, given a discrete length function $L: {\rm Mod}(R) \to \R^*$, the algebraic entropy $\ent_L$ is the unique length function $L_X : \lFin_L[X]\to \R^*$ such 
that $L_X(B(M))=L(M)$ for all $L$-finite modules $M$, where
$$B: \mod R \to \mod {R[X]}$$
is the {\em Bernoulli functor}, defined by setting $B(M)=(M^{(\N)})_\beta$, with $\beta: M^{(\N)} \to M^{(\N)}$ the right Bernoulli shift. Clearly,
the Bernoulli functor $B$ is isomorphic to the tensor product $- \otimes_R R[X]$ (see \cite{SVV}). The proof of UT was given in \cite{SVV} under the additional hypothesis that $L_X(M_\phi)=0$ for any $R[X]$-module $M_\phi$ with $M\in\Fin(L)$, and this hypothesis was removed in \cite{S}.

The proof of UT in the case of $L$ discrete is based on AT and on Proposition 3.1 of \cite{SVV}, that does not apply to a non-discrete length function. 
%
%
%
The goal of this section is to extend UT to the non-dicrete length function $L_v$ illustrated at the end of Section \ref{upp_cont_sec}.

Let us start introducing the following notions for an arbitrary ring $R$. Let $\sigma$ be a countable ascending chain of proper ideals of $R$, that is,
\begin{equation}\label{sigma}\sigma \ \ : \ \ I_0 \leq I_1 \leq I_2 \leq \cdots \leq \bigcup_n I_n =I_\infty\,,\end{equation}
and consider the functor $B_\sigma : {\rm Mod}(R) \to {\rm Mod}(R[X])$ defined on objects by:
$$\textstyle B_\sigma (M)=\left(\bigoplus_{n\geq 1} (M/MI_n)\right)_{\beta_\sigma}$$
where $\beta_\sigma : \bigoplus_{n\geq 1} (M/MI_n) \to \bigoplus_{n\geq 1} (M/MI_n)$ is the right Bernoulli shift associated with $\sigma$, defined by setting $\beta_\sigma(x_n+MI_n)_{n\geq 1}=(x_{n}+MI_{n+1})_{n\geq 0}$, where we set $x_0=0$.

\begin{lemma}\label{finiti+shif_gen}
Let $L_X:\lFin_L[X]\to \R^*$ be a length function such that 
$$L_X(B(N))=L(N)\ \ \text{ for all } \ N\in\Fin(L)\,.$$
The following statements hold true:
\begin{enumerate}[\rm (1)]
\item given $N_\phi\in \lFin_L[X]$ such that $L(N_R)<\infty$, then $L_X(N_\phi)=0$;
\item for a chain of ideals as in \eqref{sigma}, if $L(R/I_0)<\infty$ then $L_X(B_\sigma(R))=L(R/I_\infty)$.
\end{enumerate}
\end{lemma}
\begin{proof}
(1) Proceed as in \cite[Proposition 2.2]{S}.

\smallskip\noindent
(2) Let us consider first the obvious projection
$$0\to \ker(\pi)\to B(R/I_0)\overset{\pi}{\longrightarrow} B_\sigma(R)\to0\,,$$
where $\ker(\pi)$ is, as a right $R$-module, a direct sum of the form 
$$0\oplus (I_1/I_0) \oplus (I_2/I_0)\oplus \ldots \oplus (I_n/I_0)\oplus \ldots$$
with $X$ acting as a shift. For each $n\geq 1$ define the following $R[X]$-module:
$$M_n = 0\oplus (I_1/I_0)\oplus (I_2/I_0)\oplus \ldots \oplus (I_n/I_0)\oplus (I_n/I_0)\oplus\ldots\oplus (I_n/I_0)\oplus\ldots$$
with $X$ acting on $M_n$ as a shift. Since $\ker(\pi)\cong \bigcup_n M_n$, we obtain by Lemma \ref{prep} that
$$L_X(B_\sigma(R))=\lim_{n\to\infty} L_X(B(R/I_0)/M_n)\,.$$
To conclude notice that $B(R/I_0)/M_n$ differs from $B(R/I_n)$ by an $L$-finite module so that, by part (1), $L_X(B(R/I_0)/M_n)=L_X(B(R/I_n))=L(R/I_n)$ and so, using again Lemma \ref{prep}, $L_X(B_\sigma(R))=\lim_{n\to\infty}L(R/I_n)=L(R/I_\infty)$.
\end{proof}

\begin{lemma}\label{riduzione_una_dis_sui_ciclici}
Let $L_X:\lFin_L[X]\to \R^*$ be a length function such that $L_X(B(N))=L(N)$, for all $N\in\Fin(L)$. 
If 
\begin{equation}\label{cond_leq}
L_X(M_\phi)\geq \ent_L(M_\phi)\ \ \ \text{ for any cyclic $M_\phi\in\lFin_L[X]$,} 
\end{equation}
then $L_X$ coincides with the restriction of $\ent_L$ to $\lFin_L[X]$.
\end{lemma}
\begin{proof}
Given $N_\psi\in \lFin_L[X]$, we can find a continuous chain $\{(N_\alpha)_\psi:\alpha<\tau\}$ of $R[X]$-submodules of $N_\psi$ such that $N_0=0$, $(N_{\alpha+1}/N_{\alpha})_\psi$ is cyclic, and $\bigcup_{\alpha<\tau}N_\alpha=N$. Then, using additivity, sup-continuity and the condition \eqref{cond_leq}, one can prove by transfinite induction that $L_X(N_\psi)\geq \ent_L(N_\psi)$. In the same way one can show that the condition
\begin{equation}\label{cond_geq}L_X(M_\phi)\leq \ent_L(M_\phi)\ \ \ \text{ for any cyclic $M_\phi\in\lFin_L[X]$,} \end{equation}
implies that $L_X(N_\psi)\leq \ent_L(N_\psi)$ for any $N_\psi\in\lFin_L[X]$. Thus, we have just to verify that \eqref{cond_geq} always holds. Indeed, let $M_\phi=T(\phi,xR)_\phi\in\lFin_L[X]$ be cyclic and consider the following sequence of proper right ideals:
$$\sigma \  :\ \ I_0 \leq I_1 \leq I_2 \leq \cdots \leq \bigcup_n I_n=I_\infty $$
where $I_n=\{r\in R: \phi^nxr=0\}$. One can check directly that $R/I_n\cong \phi ^n xR$, so that the following map
$$\textstyle\gamma: B_\sigma(R)	\longrightarrow T(\phi, xR)_\phi \ \ \text{ such that }\ \ \gamma(r_n+I_n)_{n\geq1}=\sum_{n\geq1}\phi^n xr_n$$ 
is a well-defined and surjective homomorphism of right $R[X]$-modules, giving rise to the following short exact sequence:
$$0 \to \Ker (\gamma) \to B_\sigma(R) \to T(\phi, xR)_\phi \to 0\,.$$
By Lemma \ref{finiti+shif_gen},  $\ent_L(B_\sigma(R))=L(R/I_\infty)=L_X(B_\sigma(R))$. Furthermore, by the first part of the proof, 
$$L_X(\Ker (\gamma)) \geq \ent_L(\Ker (\gamma))$$
and so, by additivity $L_X(T(\phi, xR)_\phi) \leq  \ent_L(T(\phi, xR)_\phi)$, verifying \eqref{cond_geq}, and concluding therefore the proof.
\end{proof}

\begin{theorem}[Uniqueness Theorem]\label{UT} 
Let $R$ be an archimedean non-discrete valuation domain and let $L=L_v$ be the induced non-discrete length function on $\mod R$. There exists a unique length function $L_X : \lFin_L[X]\to \R^*$ such that, 
\begin{equation}\label{cond_ber_gen}L_X(B(N))=L(N)\ \ \text{ for all } \ N\in\Fin(L)\,.\end{equation}
Such a function coincides with the restriction of $\ent_L$ to $\lFin_L[X]$.
\end{theorem}
\begin{proof}
Notice that in this case $\lFin(L)$ coincides with the class of torsion $R$-modules. For the existence part it is enough to notice that $\ent_L$ satisfies the required properties. It remains to check uniqueness.  Indeed, let $L_X:\lFin_L[X]\to \R^*$ be a length function satisfying \eqref{cond_ber_gen}, let $M_\phi\in\lFin_L[X]$ and let us show that $L_X(M_\phi)=\ent_L(M_\phi)$. By Lemma \ref{riduzione_una_dis_sui_ciclici}, it is enough to show that 
\begin{equation}\label{cond_leq}
L_X(M_\phi)\geq \ent_L(M_\phi)\ \ \ \text{ for any cyclic $M_\phi\in\lFin_L[X]$,} 
\end{equation}
Thus, let $M_\phi=T(\phi,xR)_\phi\in\lFin_L[X]$ and consider the following sequence of proper right ideals:
$$\tau \  :\ \ J_0 \leq J_1 \leq J_2 \leq \cdots \leq \bigcup_n J_n =J_\infty,$$
where $J_n=\{r\in R: \phi^nxr\in T_n(\phi,xR)\}$. One can check directly that $R/J_n\cong T_{n+1}(\phi,xR)/T_n(\phi,xR)\cong \phi ^n xR/( \phi ^n xR \cap T_n(\phi,xR))$. Consider the following map
$$\textstyle\delta: T(\phi, xR)_\phi	\longrightarrow  B(R/J_\infty) \ \ \text{ such that }\ \ \delta\left(\sum_{n\geq 0}  \phi^n xr_n\right)=(r_n+J_\infty)_{n\geq0}\,.$$
We claim that $\delta$ is well defined. Indeed, let $xr_0+\phi x r_1+ \cdots + \phi^n x r_n=0$, then $r_n \in J_n\leq J_\infty$. By \cite[Proposition 1.6]{Z2}, $T_{n}(\phi,xR)$ is pure in $T_{n+1}(\phi,xR)$, and so there exist $s_1,\dots,s_{n-1}\in R$, such that
$$-(xr_0+\phi x r_1 + \ldots + \phi^{n-1} x r_{n-1})=\phi^n x r_n= r_n(xs_0+\phi x s_1 + \ldots + \phi^{n-1} x s_{n-1})\,.$$
But then, $x(r_0 - r_ns_0)+\phi x (r_1 - r_ns_1)+\ldots+ \phi^{n-1} x (r_{n-1} - r_ns_{n-1})=0$,
so that $r_{n-1} - r_ns_{n-1} \in J_{n-1}$. Thus, $r_{n-1} \in J_n \leq J_\infty$. Proceeding inductively this way, we can show that $r_i\in J_\infty$ for all $i=0,\dots,n$, that is, $\delta(xr_0+\phi x r_1+ \ldots + \phi^n x r_n)=0$, as claimed. Notice also that $\delta$ is a surjective homomorphism of right $R[X]$-modules, so that we obtain the following short exact sequence in $\lFin_L[X]$:
$$0 \to \Ker (\delta) \to T(\phi, xR)_\phi\to  B(R/J_\infty) \to 0\,.$$
By Lemma \ref{finiti+shif_gen}, $\ent_L(B(R/J_\infty))=L(R/J_\infty)=L_X(B(R/J_\infty))$. 
Furthermore, Proposition \ref{non-discrete} and Corollary \ref{entropia_con_potenza} imply that $\ent_L(T(\phi,xR)_{\phi})=L(R/J_\infty)$, therefore $\ent_L(\Ker (\delta))=0$, by AT; consequently we get the desired inequality: 
\begin{equation*}L_X(T(\phi, xR)_\phi ) \geq \ent_L(T(\phi, xR)_\phi)\,.\qedhere\end{equation*}
\end{proof}

\medskip\noindent
{\bf Acknowledgement.} We are indebted to Paolo Zanardo for providing us the preprint \cite{Z2}.

\end{document}